\documentclass[12pt,reqno]{amsart}
\usepackage{amsthm, amscd, amsfonts, amssymb, graphicx, color}
\usepackage[bookmarksnumbered, colorlinks, plainpages,linkcolor=green,anchorcolor=blue,citecolor=blue,urlcolor=blue]{hyperref}
\usepackage{mathrsfs}

\usepackage{hyperref}

\usepackage{amssymb}
\usepackage{tikz}
\usetikzlibrary{arrows}
\usepackage{enumerate}
\usepackage{exscale}
\usepackage{relsize}
\usepackage{pdfsync}
\numberwithin{equation}{section}

\makeatletter
\@namedef{subjclassname@2020}{%
  \textup{2020} Mathematics Subject Classification}
\makeatother

\usepackage[includemp,body={398pt,550pt},footskip=30pt,%
            marginparwidth=60pt,marginparsep=10pt]{geometry}

\newtheorem{theorem}{Theorem}[section]

\theoremstyle{definition}
\newtheorem{definition}[theorem]{Definition}

\newtheorem{remark}[theorem]{Remark}
\numberwithin{equation}{section}

\begin{document}
\title[The smallest normalized signless $\infty$-Laplacian eigenvalue]{The smallest normalized signless $\infty$-Laplacian eigenvalue for non-bipartite connected graphs}
\date{December 29, 2024.}
\author[Y. Dai]{Yi Dai}
\address{Yi Dai, School of Mathematics and Statistics, Shaanxi Normal University, Xi'an 710062, China.}
\email{daiyi@snnu.edu.cn}


\begin{abstract}
\noindent
In this paper, we aim to study the smallest normalized signless $\infty$-Laplacian eigenvalue $\mu_{\infty}$, a generalisation of the smallest signless Laplacian eigenvalue.
For a non-bipartite connected graph, we show that the invariant $\mu_{\infty}$ equals to the reciprocal of the minimal $\infty$-norm of the generalized inverses of the weighted signless incidence matrix.
An example is also given to illustrate the result.
\end{abstract}
\maketitle


\section{Introduction}
\allowdisplaybreaks[4]
The graph invariants of Laplacian matrix and signless Laplacian matrix are well studied, see \cite{2,8,12,14,15,16,17,18}. As one of typical objects in the linear spectral theory, the eigenvalues of Laplacian matrix and signless Laplacian matrix are directly related to the connectedness and bipartiteness of graph.

Let $\Omega=(V,E)$ be a non-bipartite connected graph. Note that the signless Laplacian matrix $Q$ of $\Omega$ is positive semi-definite. For $X=(x_i,v_i\in V)\in \mathbb{R}^{V}$, the quadratic form of $Q$ is $$X^{\mathrm{T}}QX=\sum_{i\sim j}(x_{i}+x_{j})^2=\parallel BX\parallel^2_2,$$
where $i\sim j$ denotes vertex $v_{i}$ is adjacent to $v_{j}$. The smallest eigenvalue of $Q$ is
$$q(\Omega)=\min_{X\in \mathbb{R}^{V}\backslash \{\mathbf{0}\}}\dfrac{X^{\mathrm{T}}QX}{\parallel X\parallel^2_2}=\min_{X\in \mathbb{R}^{V}\backslash \{\mathbf{0}\}}\dfrac{\parallel BX\parallel^2_2}{\parallel X\parallel^2_2}=\min_{\parallel X\parallel^2_2=1}\sum_{i\sim j}(x_{i}+x_{j})^2.$$
In \cite{9}, Desai and Rao showed that the smallest eigenvalue $q(\Omega)$ is bounded above and below by functions of the measure of the bipartiteness of $\Omega$.

The signless $p$-Laplacian is a nonlinear generalization of the quadratic form of the signless Laplacian matrix. For any $p\geq 1$, the smallest signless $p$-Laplacian eigenvalue
$$q_p(\Omega)=\min_{X\in \mathbb{R}^{V}\backslash \{\mathbf{0}\}}\dfrac{\parallel BX\parallel^p_p}{\parallel X\parallel^p_p}=\min_{\parallel X\parallel^p_p=1}\sum_{i\sim j}(x_{i}+x_{j})^p.$$
In \cite{4}, Borba and Schwerdtfeger proved that
$$q_\infty(\Omega)=\min_{\parallel X\parallel_\infty=1}\max_{i\sim j}|x_{i}+x_{j}|=\frac{2}{\max_{u\in V}l(u)},$$
where $l(u)$ is the length of the smallest odd cycle in graph $\Omega$.
In \cite{6}, Chang proved that the first nonzero 1-Laplacian eigenvalue equals to Cheeger's constant, which had only some upper and lower bounds in 2-Laplacian spectral theory. In \cite{7}, Chang, Shao, and Zhang introduced and developed the non-linear spectral theory of signless 1-Laplacian. They showed how the non-bipartiteness of a graph is related to the smallest normalized  signless 1-Laplacian, which can be rewritten as follows:
$$\min_{\|X\|_{1}=1}\|\bigtriangleup_{1}X\|_{1}=\min_{\|X\|_{1}=1} \sum_{i\sim j}|\frac{x_i}{d_i}+\frac{x_j}{d_j}|.$$

In this paper, for a non-bipartite connected graph $\Omega$, we study the smallest normalized  signless $\infty$-Laplacian eigenvalue
\begin{equation}\label{equ111}
\mu_{\infty}(\Omega)=\min_{\|X\|_{\infty}=1}\|\bigtriangleup_{\infty}X\|_{\infty}=
\min_{\|X\|_{\infty}=1} \max_{i\sim j} |\frac{x_i}{d_i}+\frac{x_j}{d_j}|.
\end{equation}
We provide an equivalent algebraic characterization for this invariant.
In Section \ref{sec:2}, we recall some notations and definitions.
In Section \ref{sec:3}, we prove that $\mu_{\infty}$ equals to the reciprocal  of the minimal $\infty$-norm of generalized inverses of weighted signless incidence matrix $W$ of graph, i.e.,
$$\min\limits_{\|X\|_{\infty}=1}\|WX\|_{\infty}=\frac{1}{\min\limits_{WGW=W}\|G\|_{\infty, \infty}}.$$
In Section \ref{sec:4}, we apply the result to a class of non-bipartite bicyclic connected graph.

\section{Preliminary }\label{sec:2}

In this paper, all vectors and matrices are real. We use $\mathbb{R}^{n}$ to denote the set of all $n$ dimensional column vectors over $\mathbb{R}$ and  $\mathbb{R}^{m\times n}$ to denote the set of all $m\times n$ matrices over $\mathbb{R}$.

\begin{definition}\cite{13}
The infinity norm of a vector
$x=(x_{1}, x_{2}, \cdots, x_{n})^\mathrm{T}\in\mathbb{R}^{n}$
is $\|x\|_{\infty}=\max\{|x_{i}|: i=1,2,\cdots,n\}$.
The induced matrix norm  for a real $m\times n$ matrix $A=(a_{ij})$ is $$\|A\|_{\infty, \infty}= \max\{\|AX\|_{\infty}:\|X\|_{\infty}=1\}=
\max_{1\leq i\leq m}\sum_{j=1}^{n}|a_{ij}|.$$
\end{definition}

\begin{definition}\cite{5}
Let $\Omega$ be a graph without loops or multiple edges, whose vertex set and  edge set are $V=\{v_{1}, v_{2}, \cdots, v_{n}\}$ and  $E=\{e_{1}, e_{2}, \cdots, e_{m}\}$.
Let $B=(b_{ij})_{m\times n}$ be the edge-vertex incidence matrix of $\Omega$, where
$$ b_{ij}=\left\{
\begin{array}{rcl}
	1, &      & {v_{j} \ \mathrm{is \ incident \ to} \ e_{i}, }\\
	0, &      & {\mathrm{otherwise}.}
\end{array} \right. $$
Let $D=diag(d_1, d_2, \cdots, d_n)$, where $d_i$ is the degree of vertex $v_{i}$, $i=1,2,\cdots,n$.
Define $W=BD^{-1}$ as the weighted signless edge-vertex incidence matrix of $\Omega$.
\end{definition}

\begin{remark}
By equation (\ref{equ111}), it is easy to check that the smallest normalized signless $\infty$-Laplacian eigenvalue for graph $\Omega$ is $$\mu_{\infty}(\Omega)=\min_{\|X\|_{\infty}=1}{\|WX\|_{\infty}}=
\min_{\|X\|_{\infty}=1} \max_{i\sim j}  |\frac{x_i}{d_i}+ \frac{x_j}{d_j}|,$$
where $X=(x_i,v_i\in V)\in \mathbb{R}^{V}$ and $i\sim j$ denotes vertex $v_{i}$ is adjacent to $v_{j}$.
\end{remark}

\begin{definition}\cite{3}
Let $W$ be a real $m\times n$ matrix. Then a real $n\times m$ matrix $G$ is called a generalized inverse of $W$ if $$WGW=W.$$
In particular,  if $W$ has full column rank,  then  a generalized inverse $G$  of $W$ should be left inverse of $W$,  i.e.,  $$GW=I.$$
\end{definition}

Note that there may be many generalized inverses for a matrix. The study of generalized inverses with certain  minimal norms is important in  matrix analysis, see \cite{10,11}.

\section{Main theorem and its proof}\label{sec:3}

\begin{theorem}\label{thm3.1}
Suppose that $\Omega$ is a non-bipartite connected graph with $n$ vertices and $m$ edges where $m\geq n$. Let $W$ be the weighted signless edge-vertex incidence matrix of $\Omega$. Then
$$\mu_{\infty}(\Omega)=\min\limits_{\|X\|_{\infty}=1}\|WX\|_{\infty}=\frac{1}{\min\limits_{WGW=W}\|G\|_{\infty, \infty}}.$$
\end{theorem}
\begin{proof}
By the assumption on $\Omega$,  we know $W$ has full column rank. For any generalized inverse $G\in \mathbb{R}^{n\times m}$ of $W$ satisfying $GW=I$, and any $X\in \mathbb{R}^{n}\backslash \{\mathbf{0}\}$, we have
$$\|X\|_{\infty}= \|GWX\|_{\infty} \leq \|G\|_{\infty, \infty} \cdot \|WX\|_{\infty}.$$
So
\begin{equation}\label{equ1}
\min\{\frac{\|WX\|_{\infty}}{\|X\|_{\infty}}:X\in \mathbb{R}^{n}\backslash \{\mathbf{0}\}\}\geq\frac{1}{\min\limits_{GW=I}\|G\|_{\infty, \infty}}.
\end{equation}

Now we turn to prove there exists a left inverse $G$ of $W$ such that
$$\min\{\frac{\|WX\|_{\infty}}{\|X\|_{\infty}}:X\in \mathbb{R}^{n}\backslash \{\mathbf{0}\}\}=\frac{1}{\|G\|_{\infty, \infty}}.$$
Let $\varphi:\mathbb{R}^{n}\rightarrow \mathbb{R}^{m}:X\mapsto WX$. Then  $\varphi$ is  injective. Therefore the mapping $\widetilde{\varphi}:\mathbb{R}^{n}\rightarrow Im\varphi :X\mapsto WX$ is a
linear isomorphism. Let $\widetilde{\psi}$ be the unique inverse mapping of $\widetilde{\varphi}$.
Note that there exists some $\xi=(\xi_{1}, \xi_{2}, \cdots, \xi_{n})^{T} \in \mathbb{R}^{n} \backslash\{\mathbf{0}\}$ such that
$$c=\frac{\|\varphi(\xi)\|_{\infty}}{\|\xi\|_{\infty}}=\min\{\frac{\|\varphi(X)\|_{\infty}}{\|X\|_{\infty}}:X\in\mathbb{R}^{n}\backslash \{\mathbf{0}\}\}.$$
Take $\eta=\varphi(\xi)$, then$c= \frac{\|\eta\|_{\infty}}{\|\xi\|_{\infty}}$.
Since $\widetilde{\psi}(\eta)=\xi$, by definition of $c$,  we can get
$$\max_{Y\in Im\varphi}\frac{\|\widetilde{\psi}(Y)\|_{\infty}}{\|Y\|_{\infty}}=\frac{\|\widetilde{\psi}(\eta)\|_{\infty}}{\|\eta\|_{\infty}}=\frac{\|\xi\|_{\infty}}{\|\varphi(\xi)\|_{\infty}}
=\frac{1}{c},$$
i.e., $\|\widetilde{\psi}\|_{\infty, \infty}=\frac{1}{c}$.

Write
\begin{align*}
\widetilde{\psi}:&Im\varphi \rightarrow  \mathbb{R}^{n}\\
&\alpha\mapsto (\widetilde{\psi}_{1}(\alpha), \widetilde{\psi}_{2}(\alpha), \cdots, \widetilde{\psi}_{n}(\alpha))^{T}.
\end{align*}
By Hahn-Banach Theorem (see \cite{1}, P. 34), there exists a norm-preserving linear extension
$\psi_{i}: \mathbb{R}^{m}  \rightarrow \mathbb{R}$ of $\widetilde{\psi}_{i}$,   where $i=1, 2, \cdots, n$.
Take
\begin{align*}
\psi:&\mathbb{R}^{m} \rightarrow  \mathbb{R}^{n}\\
&\beta\mapsto (\psi_{1}(\beta), \psi_{2}(\beta), \cdots, \psi_{n}(\beta))^{T}.
\end{align*}
We can get
\begin{align*}
\|\psi\|_{\infty, \infty}& =\max_{Z \in \mathbb{R}^{m} \backslash\{\mathbf{0}\}}\frac{\|\psi(Z)\|_{\infty}}{\|Z\|_{\infty}}\\
&=\max_{ i=1, 2, \cdots, n}\max_{Z \in \mathbb{R}^{m}\backslash\{\mathbf{0}\}}\frac{|\psi_{i}(Z)|}{\|Z\|_{\infty}}\\
&=\max_{ i=1, 2, \cdots, n}\max_{Z\in Im\varphi\backslash\{\mathbf{0}\}}\frac{|\widetilde{\psi}_{i}(Z)|}{\|Z\|_{\infty}}\\
&=\max_{ i=1, 2, \cdots, n}\frac{|\widetilde{\psi}_{i}(\eta)|}{\|\eta\|_{\infty}}\\
&=\|\widetilde{\psi}\|_{\infty, \infty}.
\end{align*}
So $\psi$ is a norm-preserving linear linear extension of $\widetilde{\psi}$, and thus there exists some generalized inverse matrix $G$ of $W$ such that
\begin{equation}\label{equ2}
\|G\|_{\infty, \infty}=\|\psi\|_{\infty, \infty} =\frac{1}{c}=\frac{1}{\min\limits_{\|X\|_{\infty}=1}\|WX\|_{\infty}}.
\end{equation}

Combining equations (\ref{equ1}) and (\ref{equ2}),  we get
$$\mu_{\infty}(\Omega)=\min\limits_{\|X\|_{\infty}=1}\|WX\|_{\infty}=\frac{1}{\min\limits_{WGW=W}\|G\|_{\infty, \infty}}.$$
\end{proof}

\section{An example}\label{sec:4}

As an example, we introduce the definition for the class of non-bipartite bicyclic connected graphs $ET(n, a, b)$ in Definition \ref{def41}.
We compute the minimal $\infty$-norm of generalized inverses of the weighted signless incidence matrices $W$ for $ET(n, a, b)$ in Theorem \ref{thm4.1}. And we construct a vector $Y\in \mathbb{R}^{n} \backslash \{\mathbf{0}\}$ in Remark \ref{rem4.2} such that $$\mu_{\infty}=\frac{{\|WY\|_{\infty}}}{\|Y\|_{\infty}}
=\min\limits_{\|X\|_{\infty}=1}\|WX\|_{\infty}$$
to illustrate the result of Theorem \ref{thm3.1}.

\begin{definition}\label{def41}
Suppose that $n=a+b-1$, where integers $a, b>2$ and $a$ is odd. Let $V=\{v_{1}, v_{2}, \cdots, v_{n}\}$ and $E=\{e_{1}, e_{2}, \cdots, e_{n+1}\}$, where

(1) $e_{k}$ is an edge between $v_{k}$ and $v_{k+1}$,  $k=1, 2, \cdots, a-1$,

(2) $e_{a}$ is an edge between $v_{a}$ and $v_{1}$,

(3) $e_{k}$ is an edge between $v_{k-1}$ and $v_{k}$, $k=a+1, a+2, \cdots, n$,

(4) $e_{n+1}$ is an edge between $v_{n}$ and $v_{a}$.\\
We denote the bicyclic graph $(V,E)$ by $ET(n, a, b)$.
\end{definition}

 For example, $ET(9,3,7)$ is as FIGURE \ref{fig1}.
\begin{figure}[!htbp]
\begin{center}
\tikzstyle{place}=[circle,draw,inner sep=0pt,minimum size=5pt]
\begin{tikzpicture}[scale=4]

\fill  (0,0.6) circle (0.3pt);
\node[left] at (0,0.6) {$v_{1}$};
\fill  (0,0.2) circle (0.3pt);
\node[left] at (0,0.2) {$v_{2}$};
\fill  (0.4,0.4) circle (0.3pt);
\node[right] at (0.4,0.4) {$v_{3}$};
\fill  (0.6,0.8) circle (0.3pt);
\node[above] at (0.6,0.8) {$v_{4}$};
\fill  (1.0,0.8) circle (0.3pt);
\node[above] at (1.0,0.8) {$v_{5}$};
\fill  (1.2,0.6) circle (0.3pt);
\node[right] at (1.2,0.6) {$v_{6}$};
\fill  (1.2,0.2) circle (0.3pt);
\node[right] at (1.2,0.2) {$v_{7}$};
\fill  (1.0,0) circle (0.3pt);
\node[below] at (1.0,0) {$v_{8}$};
\fill  (0.6,0) circle (0.3pt);
\node[below] at (0.6,0) {$v_{9}$};

\draw(0,0.6)--(0,0.2) node[pos=0.5,left] {$e_{1}$};
\draw(0,0.2)--(0.4,0.4)node[pos=0.5,below] {$e_{2}$};
\draw(0.4,0.4)--(0,0.6)node[pos=0.5,above] {$e_{3}$};
\draw(0.4,0.4)--(0.6,0.8) node[pos=0.7,left] {$e_{4}$};
\draw(0.6,0.8)--(1.0,0.8)node[pos=0.5,above] {$e_{5}$};
\draw(1.0,0.8)--(1.2,0.6)node[pos=0.4,right] {$e_{6}$};
\draw(1.2,0.6)--(1.2,0.2)node[pos=0.5,right] {$e_{7}$};
\draw(1.2,0.2)--(1.0,0) node[pos=0.6,right] {$e_{8}$};
\draw(1.0,0)--(0.6,0)node[pos=0.5,below] {$e_{9}$};
\draw(0.6,0)--(0.4,0.4)node[pos=0.3,left] {$e_{10}$};
\end{tikzpicture}
\caption{$ET(9,3,7)$}
\label{fig1}
\end{center}
\end{figure}

\begin{theorem} \label{thm4.1}
Let $W$ be the weighted signless incidence matrix of $ET(n, a, b)$ and $G=(g_{ij})$ be a generalized inverse of $W$. Then we have
$$\min_{WGW=W}\|G\|_{\infty, \infty}=\left\{
\begin{array}{lcl}
a,      &      & {\mathrm{if} \  b \ \mathrm{is} \ \mathrm{odd} \ \mathrm{and} \ a>2b, }\\
2\min\{a, b\},      &      & {\mathrm{if} \ b \ \mathrm{is} \ \mathrm{odd} \ \mathrm{and}  \ a\leq 2b,  \ b\leq 2a, }\\
b,      &      & {\mathrm{if}  \ b \ \mathrm{is} \ \mathrm{odd} \ \mathrm{and}  \ b>2a, }\\
\max\{2a, a+b\},      &      & {\mathrm{if} \  b \ \mathrm{is} \ \mathrm{even}.}
\end{array} \right. $$
\end{theorem}
\begin{proof}

The weighted signless incidence matrix $W$ of $ET(n, a, b)$ is
$$\bordermatrix{
  & 1 &   &   & \cdots &   & a &   &   & \cdots &   & n\cr
1 & \frac{1}{2} & \frac{1}{2} & 0 & \cdots & 0 & 0 & 0 & 0 & \cdots & 0 & 0\cr
  & 0 & \frac{1}{2} & \frac{1}{2} & \cdots & 0 & 0 & 0 & 0 & \cdots & 0 & 0\cr
  & 0 & 0 & \frac{1}{2} & \cdots & 0 & 0 & 0 & 0 & \cdots & 0 & 0\cr
\vdots & \vdots & \vdots & \vdots & \ddots & \vdots & \vdots & \vdots & \vdots & \ddots & \vdots & \vdots\cr
  & 0 & 0 & 0 & \cdots & \frac{1}{2} & \frac{1}{4} & 0 & 0 & \cdots & 0 & 0\cr
a & \frac{1}{2} & 0 & 0 & \cdots & 0 & \frac{1}{4} & 0 & 0 & \cdots & 0 & 0\cr
  & 0 & 0 & 0 & \cdots & 0 & \frac{1}{4} & \frac{1}{2} & 0 & \cdots & 0 & 0\cr
  & 0 & 0 & 0 & \cdots & 0 & 0 & \frac{1}{2} & \frac{1}{2} & \cdots & 0 & 0\cr
\vdots & \vdots & \vdots & \vdots & \ddots & \vdots & \vdots & \vdots & \vdots & \ddots & \vdots & \vdots\cr
 & 0 & 0 & 0 & \cdots & 0 & 0 & 0 & 0 & \cdots & \frac{1}{2} & \frac{1}{2}\cr
n+1 & 0 & 0 & 0 & \cdots & 0 & \frac{1}{4} & 0 & 0 & \cdots & 0 & \frac{1}{2}\cr
}.$$
Since $W$ has full column rank, so its generalized inverse $G$ satisfying $GW=I$.\\
\textbf{Case 1:}
 If $b$ is odd,  then let $\hat{G}$ be
$$
\bordermatrix{
&  1 &   &   &  & \cdots &  &   & a &   &   &  & \cdots &   & n+1\cr
1 & 1 & -1 & 1 & -1 & \cdots & 1 & -1 & 1 & 0 & 0 & 0 &\cdots & 0 & 0\cr
&1 & 1 & -1 & 1& \cdots &-1 & 1 & -1 & 0 & 0 & 0&\cdots & 0 & 0\cr
&-1 & 1 & 1 &-1& \cdots &1& -1 & 1 & 0 & 0 &0 & \cdots & 0 & 0\cr
&1 & -1 & 1 & 1&\cdots & -1 & 1 & -1 & 0 & 0 & 0& \cdots & 0 & 0\cr
\vdots &\vdots & \vdots & \vdots& \vdots & \ddots & \vdots & \vdots & \vdots & \vdots & \vdots & \vdots & \ddots & \vdots & \vdots\cr
&-1 & 1 & -1 & 1& \cdots &1 & -1 & 1 & 0 & 0 &0& \cdots & 0 & 0\cr
&1 & -1 & 1 & -1 &\cdots &1 & 1 & -1 & 0 & 0 &0 & \cdots & 0 & 0\cr
a&-2 & 2 & -2 &2& \cdots & -2 & 2 & 2 & 0 & 0 & 0&\cdots & 0 & 0\cr
&1 & -1 & 1 & -1 &\cdots &1 & -1 & -1 & 2 & 0 & 0&\cdots & 0 & 0\cr
&-1 & 1 & -1 & 1 &\cdots & -1 & 1 & 1 & -2 & 2 & 0&\cdots & 0 & 0\cr
\vdots &\vdots & \vdots & \vdots& \vdots & \ddots & \vdots & \vdots & \vdots & \vdots & \vdots & \vdots & \ddots & \vdots & \vdots\cr
& 1 & -1 & 1 & -1& \cdots &1& -1 & -1 & 2 & -2 &2 & \cdots & 0 & 0\cr
n&-1 & 1 & -1 &1 & \cdots&-1 & 1 & 1 & -2 & 2 & -2&\cdots & 2 & 0\cr
}.$$
We may check that $\hat{G}$ is a generalized inverse matrix of $W$.
Any generalized inverse $G$ of $W$  can be expressed as
$$G=\hat{G}+\alpha\beta,$$
where $\alpha=(z_{1}, z_{2}, \cdots, z_{n})^{\mathrm{T}}$ is an arbitrary vector in $\mathbb{R}^n$,  and
$$\beta=
\bordermatrix{
&1 &   &   &   & \cdots &   &   & a &   &   &   & \cdots &   & n+1\cr
&1 & -1 & 1 & -1 & \cdots & 1 & -1 & -1 & 1 & -1 & 1 & \cdots & -1 & 1\cr
}.$$

As we know,  if $x_i$ is the median of a sequence of real numbers $${m_{i1},m_{i2},\cdots,m_{in}},$$ then $$\sum_{j=1}^{n}|m_{ij}-x_i|=\min\{\sum_{j=1}^{n}|m_{ij}-y|:y\in\mathbb{R}\}.$$

Since $$\min_{WGW=W}\|G\|_{\infty, \infty}=\min_{\alpha}\|G\|_{\infty, \infty}  =\min_{\alpha}\max_{1\leq i\leq n}\sum_{j=1}^{n+1}|g_{ij}|
= \max_{1\leq i\leq n}\min_{\alpha}\sum_{j=1}^{n+1}|g_{ij}|,$$
where $\alpha$ is an arbitrary vector in $\mathbb{R}^n$.
Now we compute $$\min_{\alpha}\sum_{j=1}^{n+1}|g_{ij}|$$ for each row.

(1) For $i=1, 2, \cdots, a-1$,
  let $$s=\left\{
\begin{array}{lcl}
i,      &      & {\mathrm{if} \ i \ \mathrm{is} \ \mathrm{even}, }\\
a-i,      &      & {\mathrm{if} \ i \ \mathrm{is} \ \mathrm{odd}, }
\end{array} \right. $$ and $t=a-s$.
Then $$\sum\limits_{j=1}^{n+1}|g_{ij}|=|-1-z_{i}|\cdot s+|1-z_{i}|\cdot t+|z_{i}|\cdot b.$$
\begin{enumerate}[(i)]
  \item If $s>b+t$,  we take $z_{i}=-1$, then $\min\limits_{\alpha}\sum\limits_{j=1}^{n+1}|g_{ij}|=2t+b<s+t=a.$
  \item If $t>b+s$,  we take $z_{i}=1$, then $\min\limits_{\alpha}\sum\limits_{j=1}^{n+1}|g_{ij}|=2s+b<s+t=a.$
  \item If $s\leq b+t$ and $t\leq b+s$,  we take $z_{i}=0$, then $\min\limits_{\alpha}\sum\limits_{j=1}^{n+1}|g_{ij}|=s+t=a.$
\end{enumerate}
(2) For $i=a$, we have
$$\sum\limits_{j=1}^{n+1}|g_{ij}|=|2-z_{i}|\cdot a+|z_{i}|\cdot b.$$
\begin{enumerate}[(i)]
  \item If $a>b$,  we take $z_{i}=2$, then  $\min\limits_{\alpha}\sum\limits_{j=1}^{n+1}|g_{ij}|=2b.$
  \item If $a=b$,  we take $z_{i}=1$, then  $\min\limits_{\alpha}\sum\limits_{j=1}^{n+1}|g_{ij}|=a+b.$
  \item If $a<b$,  we take $z_{i}=0$, then  $\min\limits_{\alpha}\sum\limits_{j=1}^{n+1}|g_{ij}|=2a.$
\end{enumerate}
(3) For $i=a+1, a+3, \cdots, a+b-2$, let $s=i-a$ and $t=a+b-i$. Then
$$\sum\limits_{j=1}^{n+1}|g_{ij}|=|-1-z_{i}|\cdot a+|-2-z_{i}|\cdot s+|z_{i}|\cdot t.$$
\begin{enumerate}[(i)]
  \item If $s>a+t$,  we take $z_{i}=-2$, then $\min\limits_{\alpha}\sum\limits_{j=1}^{n+1}|g_{ij}|=a+2t<s+t=b.$
  \item If $t>a+s$,  we take $z_{i}=0$, then $\min\limits_{\alpha}\sum\limits_{j=1}^{n+1}|g_{ij}|=a+2s<s+t=b.$
  \item  If $s\leq a+t$ and $t\leq a+s$,  we take $z_{i}=-1$, then  $\min\limits_{\alpha}\sum\limits_{j=1}^{n+1}|g_{ij}|=s+t=b.$
\end{enumerate}
(4) For $i=a+2, a+4, \cdots, a+b-1$,  let $s=i-a$ and $t=a+b-i$. Then
$$\sum\limits_{j=1}^{n+1}|g_{ij}|=|1-z_{i}|\cdot a+|2-z_{i}|\cdot s+|z_{i}|\cdot t.$$
\begin{enumerate}[(i)]
  \item If $s>a+t$,  we take $z_{i}=2$, then $\min\limits_{\alpha}\sum\limits_{j=1}^{n+1}|g_{ij}|=a+2t<s+t=b.$
  \item If $t>a+s$,  we take $z_{i}=0$, then $\min\limits_{\alpha}\sum\limits_{j=1}^{n+1}|g_{ij}|=a+2s<s+t=b.$
  \item If $s\leq a+t$ and $t\leq a+s$,  we take $z_{i}=1$, then  $\min\limits_{\alpha}\sum\limits_{j=1}^{n+1}|g_{ij}|=s+t=b.$
\end{enumerate}

So when $b$ is odd,  $$\min_{WGW=W}\|G\|_{\infty, \infty}=\left\{
\begin{array}{lcl}
a,      &      & {\mathrm{if} \ a>2b, }\\
2\min\{a, b\},      &      & {\mathrm{if} \ a\leq 2b \ \mathrm{and} \ b\leq 2a, }\\
b,      &      & {\mathrm{if} \ b>2a.}\\
\end{array} \right. $$ The proof in Case 1 is complete.\\
\textbf{Case 2:} If $b$ is even,  then let $\hat{G}$ be
$$
\bordermatrix{
&  1 &   &   &  & \cdots &  &   & a &   &   &  & \cdots &   & n+1\cr
1 & 1 & -1 & 1 & -1 & \cdots & 1 & -1 & 1 & 0 & 0 & 0 &\cdots & 0 & 0\cr
&1 & 1 & -1 & 1& \cdots &-1 & 1 & -1 & 0 & 0 & 0&\cdots & 0 & 0\cr
&-1 & 1 & 1 &-1& \cdots &1& -1 & 1 & 0 & 0 &0 & \cdots & 0 & 0\cr
&1 & -1 & 1 & 1&\cdots & -1 & 1 & -1 & 0 & 0 & 0& \cdots & 0 & 0\cr
\vdots &\vdots & \vdots & \vdots& \vdots & \ddots & \vdots & \vdots & \vdots & \vdots & \vdots & \vdots & \ddots & \vdots & \vdots\cr
&-1 & 1 & -1 & 1& \cdots &1 & -1 & 1 & 0 & 0 &0& \cdots & 0 & 0\cr
&1 & -1 & 1 & -1 &\cdots &1 & 1 & -1 & 0 & 0 &0 & \cdots & 0 & 0\cr
a&-2 & 2 & -2 &2& \cdots & -2 & 2 & 2 & 0 & 0 & 0&\cdots & 0 & 0\cr
&1 & -1 & 1 & -1 &\cdots &1 & -1 & -1 & 2 & 0 & 0&\cdots & 0 & 0\cr
&-1 & 1 & -1 & 1 &\cdots & -1 & 1 & 1 & -2 & 2 & 0&\cdots & 0 & 0\cr
&1 & -1 & 1 & -1 &\cdots &1 & -1 & -1 & 2 & -2 & 2&\cdots & 0 & 0\cr
\vdots &\vdots & \vdots & \vdots& \vdots & \ddots & \vdots & \vdots & \vdots & \vdots & \vdots & \vdots & \ddots & \vdots & \vdots\cr
&-1 & 1 & -1 & 1 &\cdots & -1 & 1 & 1 & -2 & 2 & -2&\cdots & 0 & 0\cr
n&1 & -1 & 1 & -1 &\cdots &1 & -1 & -1 & 2 & -2 & 2&\cdots & 2 & 0\cr
}.$$
We may check that $\hat{G}$ is  a generalized inverse matrix of $W$.
Any generalized inverse of $W$  can be expressed as
$$G=\hat{G}+\alpha\beta,$$
where $\alpha=(z_{1}, z_{2}, \cdots, z_{n})^{\mathrm{T}}$ is an arbitrary vector in $\mathbb{R}^n$,  and
$$\beta=
\bordermatrix{
&1 &   & \cdots &   & a &   &   &  & &\cdots &   & n+1\cr
&0 & 0 & \cdots & 0 & 0 & -1 & 1 & -1& 1&\cdots & -1 & 1\cr
}$$

Similarly as in Case 1, now we can get $$\min_{WGW=W}\|G\|_{\infty, \infty}=\min_{\alpha}\|G\|_{\infty, \infty}=\max\{2a, a+b\}.$$ The proof in Case 2 is complete.
\end{proof}

Now we construct a vector $Y\in \mathbb{R}^{n} \backslash \{\mathbf{0}\}$ such that the smallest normalized signless $\infty$-Laplacian eigenvalue
$$\mu_{\infty}=\frac{{\|WY\|_{\infty}}}{\|Y\|_{\infty}}$$
for the non-bipartite bicyclic connected graphs $ET(n, a, b)$.

\begin{remark}\label{rem4.2}
Let $W$ be the weighted signless incidence matrix of $ET(n, a, b)$.
We construct the optimal vector
$Y=(y_{1}, \ y_{2}, \ \cdots, \ y_{n})^{\mathrm{T}}\in \mathbb{R}^{n} \backslash\{\mathbf{0}\}$
as follows, such that
$$\mu_{\infty}=\frac{{\|WY\|_{\infty}}}{\|Y\|_{\infty}}=
\frac{\max\limits_{i\sim j}|\frac{y_i}{d_i}+ \frac{y_j}{d_j}|}{\max\limits_{i}  |y_i|}=\min\limits_{\|X\|_{\infty}=1}\|WX\|_{\infty}=
\frac{1}{\min\limits_{GW=I}\|G\|_{\infty,\infty}}.$$

If $a$ and $b$ are odd, without loss of generality, we only need consider two cases when $a\geq b$ as follows. If $a$, $b$ are both odd, and $b\geq a$, we can easily get the vector $Y$ by exchanging $a$ and $b$ in Case 1 and Case 2.\\
\textbf{Case 1:} If $b$ is odd, and $a>2b$. Suppose that
$Y=(y_{i})\in \mathbb{R}^{n} \backslash \{\mathbf{0}\}$ satisfying
$$ y_{i}=\left\{
\begin{array}{lcl}
(-1)^{i}\cdot(a-2\cdot|\frac{a-1}{2}-i|),     &      & {\mathrm{if} \ i= 1,2,\cdots,a-1,}\\
2,     &      & {\mathrm{if} \ i=a,}\\
(-1)^{i}\cdot(b-2\cdot|a+\frac{b+1}{2}-i|),     &      & {\mathrm{if} \ i= a+1,a+2,\cdots,a+b-1.}\\
\end{array} \right. $$
Then $$\max\limits_{i}  |y_i|=|y_{\frac{a-1}{2}}|=a,$$ $$\max\limits_{i\sim j}  |\frac{y_i}{d_i}+ \frac{y_j}{d_j}|=1.$$ So  $$\mu_{\infty}=\frac{{\|WY\|_{\infty}}}{\|Y\|_{\infty}}=
\frac{\max\limits_{i\sim j}  |\frac{y_i}{d_i}+ \frac{y_j}{d_j}|}{\max\limits_{i}  |y_i|}=\frac{1}{a}.$$\\
\textbf{Case 2:} If $b$ is odd, and $b\leq a\leq 2b$. Suppose that
$Y=(y_{i})\in \mathbb{R}^{n} \backslash \{\mathbf{0}\}$ satisfying
$$y_{i}=\left\{
\begin{array}{lcl}
(-1)^{i}\cdot(b-2i),     &      & {\mathrm{if} \ i= 1,2,\cdots,\frac{b-1}{2},}\\
(-1)^{\frac{b-1}{2}},     &      & {\mathrm{if} \ i= \frac{b+1}{2},\frac{b+1}{2}+1,\cdots,a-\frac{b+1}{2},}\\
(-1)^{a-i}\cdot(b-2(a-i)),     &      & {\mathrm{if} \ i= a-\frac{b-1}{2},a-\frac{b-1}{2}+1,\cdots,a-1,}\\
2b,     &      & {\mathrm{if} \ i=a,}\\
(-1)^{i-a}\cdot[b-2(i-a)],     &      & {\mathrm{if} \ i= a+1,a+2,\cdots,a+\frac{b-1}{2},}\\
(-1)^{a+b-i}\cdot[b-2(a+b-i)],     &      & {\mathrm{if} \ i= a+\frac{b+1}{2},a+\frac{b+1}{2}+1,\cdots,a+b-1.}\\
\end{array} \right. $$
Then $$\max\limits_{i}  |y_i|=|y_a|=2b,$$ $$\max\limits_{i\sim j}  |\frac{y_i}{d_i}+ \frac{y_j}{d_j}|=1.$$
So
$$\mu_{\infty}=\frac{{\|WY\|_{\infty}}}{\|Y\|_{\infty}}=
\frac{\max\limits_{i\sim j}  |\frac{y_i}{d_i}+ \frac{y_j}{d_j}|}{\max\limits_{i}  |y_i|}=\frac{1}{2b}.$$\\
\textbf{Case 3:} If $a$ is odd and $b$ is even, suppose that
the vector $Y\in \mathbb{R}^{n} \backslash \{\mathbf{0}\}$ satisfying
$$ y_{i}=\left\{
\begin{array}{lcl}
(-1)^{i}\cdot(a-2i),     &      & {\mathrm{if} \ i= 1,2,\cdots,\frac{a-1}{2},}\\
(-1)^{a-i}\cdot[a-2\cdot(a-i)],     &      & {\mathrm{if} \ i= \frac{a+1}{2},\frac{a+1}{2}+1,\cdots,a-1,}\\
2a,     &      & {\mathrm{if} \ i=a,}\\
(-1)^{i-a}\cdot[a+2\cdot(i-a)],     &      & {\mathrm{if} \ i= a+1,a+2,\cdots,a+\frac{b}{2},}\\
(-1)^{a+b-i}\cdot[a+2\cdot(a+b-i)],     &      & {\mathrm{if} \ i= a+\frac{b}{2}+1,a+\frac{b}{2}+2,\cdots,a+b-1.}\\
\end{array} \right. $$

When $a>b$, we have $$\max\limits_{i}  |y_i|=|y_a|=2a,$$
 $$\max\limits_{i\sim j}  |\frac{y_i}{d_i}+ \frac{y_j}{d_j}|=1.$$ So  $$\mu_{\infty}=\frac{{\|WY\|_{\infty}}}{\|Y\|_{\infty}}=
\frac{\max\limits_{i\sim j}  |\frac{y_i}{d_i}+ \frac{y_j}{d_j}|}{\max\limits_{i}  |y_i|}=\frac{1}{2a}.$$

When $a<b$, we have $$\max\limits_{i}  |y_i|=|y_{a+\frac{b}{2}}|=a+b,$$ $$\max\limits_{i\sim j}  |\frac{y_i}{d_i}+ \frac{y_j}{d_j}|=1.$$ So $$\mu_{\infty}=\frac{{\|WY\|_{\infty}}}{\|Y\|_{\infty}}=
\frac{\max\limits_{i\sim j}  |\frac{y_i}{d_i}+ \frac{y_j}{d_j}|}{\max\limits_{i} |y_i|}=\frac{1}{a+b}.$$

The computational results from these construction and Theorem \ref{thm4.1} coincide. These verify the identity in Theorem \ref{thm3.1} for $ET(n,a,b)$.
\end{remark}

\medskip

\noindent{\bf Conflict of interest:} The authors declare there is no conflict of interest.

\noindent{\bf Funding:} The research is supported by the Fundamental Research Funds for the Central Universities of China (Grant no. GK202304003).

\end{document}